\documentclass[12 pt]{amsart}
\usepackage{amsmath, amsthm, amsfonts, amssymb,mathrsfs,multirow,graphicx}
\theoremstyle{plain}
\newtheorem{example}{Example}
\newtheorem{thm}{Theorem}
\newtheorem*{thmA}{Theorem A}

\newtheorem{remark}{Remark}

\newtheorem{lemma}{Lemma}
\newtheorem{prop}{Proposition}

\title[Julia limiting directions]{Large scaled geometry of Julia sets of entire and meromorphic functions}
\begin{document}

\author{Jun Wang}
\address{School of Mathematical Sciences, Fudan University,
Shanghai 200433, P.~R.\ China}
\email{majwang@fudan.edu.cn}
\author{Xiao Yao}
\address{Shanghai Center of Mathematical Sciences, Fudan University, Shanghai 200433, P.~R.\ China}
\email{yaoxiao@fudan.edu.cn}

\date{}
\begin{abstract}
  In this paper, we study  the large scaled geometric structure of  Julia sets of entire and meromorphic functions.  Roughly speaking, the structure gives us some asymptotic information about the Julia set near the essential singularity. We will show that one part of this structure is determined by the transcendental directions coming from function theoretic point of view.
\end{abstract}
\maketitle
\section{Introduction and Main Results}
 Let $f$ be a transcendental entire or meromorphic in the complex plane $\mathbb{C}$, and in this paper, the meromorphic function has at least one pole. The Fatou set
$\mathcal{F}(f)$ is the set of points $z$ such that the iterates
$f^n(z)(n=1,2,\cdots)$ of $f$ are well defined and $\{f^n\}$ forms a normal family in some neighborhood of $z$, and the Julia set $\mathcal{J}(f)$ is its complement. The iteration on the Julia set is usually quite complicated,  indeed for $z\in \mathcal{J}(f)$ and $U$ is any neighborhood of $z$, then by Montel's Theorem, $\bigcup f^n(U)$ contains all points in $\mathbb{C}$ with at most one exception. Some basic knowledge  and recent progress in transcendental iteration theory could be found in the survey papers \cite{EL-Survey,ber1,sch, Eremenko} and reference therein. \vskip 2mm
\par We know that from a local point of view, the Julia set has a delicate and self similar  structure in some sense. Now let's imagine ourselves to stand at the north pole of Riemann sphere, that is $\infty$, the essentially singularity
for transcendental entire and meromorphic functions. A natural question is what is the behavior or geometric property of Julia set
near $\infty$. To get a rough impression, we would ignore the local structure, and depict the Julia set in a large scale. In order to do so,
we use the value distribution theory as an important tool. The usual notations and basic results of this theory can be found in \cite{gol1,hay1}. For example, $T(r, f)$ and $N(r, f)$ denote the Nevanlinna characteristic
function and the integrated counting function of poles with respect to $f$, respectively.\vskip 2mm
\par There are already some references in discussing the structure of Julia set around $\infty$. Baker \cite{bak1} proved that $\mathcal{J}(f)$ can not be contained in any finite union of straight lines if $f$ is a transcendental entire function. While it fails for transcendental meromorphic functions, see tangent map as an example since $\mathcal{J}(\tan z)=\mathbb{R}$.
From the viewpoint of angular distribution, Qiao \cite{qiao2} introduced the limiting direction of $\mathcal{J}(f)$. For the brevity, the limiting direction of Julia set is called Julia limiting direction in this paper. A value $\theta\in[0,2\pi)$ is said to be a Julia limiting direction if there is an unbounded sequence $\{z_n\}\subset\mathcal{J}(f)$ such that
\[\lim_{n\rightarrow\infty}\arg z_n=\theta.\]
\par We use $L(f)$ to denote the set of all Julia limiting directions of $f$. Clearly, if $f$ is transcendental, $L(f)$ is non-empty and closed in $[0, 2\pi)$, and reveals the structure of large scaled geometry of Julia set. Note that $\theta\in[0,2\pi)$ can be seen as the argument of one ray originating from $0$, we will identify the two endpoints of $[0, 2\pi)$ to make it as a compact set throughout the  paper. \vskip 2mm
\par For transcendental entire functions, Qiao \cite{qiao2} noticed that there is a relation between Lebesgue measure $meas(L(f))$ and the growth order of $f$, where the order $\rho(f)$ and the lower order $\mu(f)$ are defined respectively as
$$\rho(f)=\limsup_{r\to\infty}\frac{\log^+ T(r,f)}{\log r},\quad \mu(f)=\liminf_{r\to\infty}\frac{\log^+ T(r,f)}{\log r},$$
where $\log^{+}x=\max\{\log x, 0\}$ for any $x>0$. In fact, Qiao proved the following theorem, and remarked that the below estimate is sharp by modifying the function in Mittag-Leffler class.
 \begin{thmA}\cite{qiao2}
 Let $f$ be a transcendental entire function of lower order $\mu<\infty$. Then there exists a closed interval
 $I\subseteq L(f)$ such that
 \[\label{Qiao-lower-bound}
 meas(I)\geq \min\{2\pi,\pi/\mu\}.
 \]
\end{thmA}
 Moreover, there exists the entire function of infinite order growth, such that $L(f)$ consists of only one limiting direction \cite{bak1}. Later, Theorem A was generalized to  meromorphic functions under certain condtions, see \cite{zheng-wang-huang, qiu-wu} for the details.\vskip 2mm
\par In this paper, we mainly study large scaled geometric structure of Julia set from two aspects. One aspect is to solve the inverse problem of Julia limiting directions, that is, to construct entire or meromorphic functions with a preassigned set
of Julia limiting directions. The other respect is to find an important subset of $L(f)$ which often can be easily determined and which is stable under small perturbation. In fact, these two respects are closely related, since we will use this subsect in the study of the inverse problem. \vskip 2mm
\par The following two theorems enable us
to partially answer the inverse problem.  \vskip 0.5mm
\begin{thm}\label{theorem-infinite-order}
Suppose that $\mathcal{E}$ is a compact subset of $[0,2\pi)$, and that $\rho\in[0,\infty]$. Then there is a transcendental entire function $f$, of infinite lower order, and a transcendental meromorphic function $g$, of order $\rho$, such that
$$L(f)=L(g)=\mathcal{E}.$$
\end{thm}
\vskip 0.5mm
\par We know $L(f)=[0, 2\pi)$ for transcendental entire $f$ of order $\rho(f)\in[0,1/2]$. Thus, for the above inverse problem on entire functions of finite lower order, we only need to consider entire functions with order more than $1/2$. We can not solve it completely, but the strategy in our proof of Theorem \ref{theorem-infinite-order} can be used to deal with some partial case, and the corresponding result
is stated below.

\begin{thm}\label{theorem-finite-order}
Suppose that $\rho\in (1/2,\infty)$, positive integer $m\leq 2\rho$, and that all $I_i(i=1,\cdots,m)$ are finitely many disjoint closed intervals with $meas(I_i)\geq \pi/\rho$. Then there always exists an entire function $f$ of order $\rho$ such that
  $$L(f)=\bigcup_{i=1}^{m} I_i.$$
\end{thm}

\par Theorem \ref{theorem-infinite-order} and Theorem \ref{theorem-finite-order} could be proved by some advanced machineries by the application of  approximation theory in \cite{gol1}, or the techniques of quasi-conformal folding developed in \cite{bishop}. Here, in this paper, we will stress a relatively simple method. Roughly speaking, it should be described as the a kind of soft interpolation technique without  quasiconformal surgery.

Generally, due to the complicated geometry of Julia set, it is difficult to detect possible Julia limiting directions. Our next result shows that $L(f)$ has an important and easily determined subset $\mathcal{TD}(f)$, which is the union of all transcendental directions. Here, a value $\theta\in[0,2\pi)$ is said to be a transcendental direction of $f$ if there exists an unbounded sequence of $\{z_{n}\}$
such that
\[\lim\limits_{n\rightarrow\infty}\arg z_{n}=\theta,\quad \lim\limits_{n\rightarrow\infty}\frac{\log|f(z_n)|}{\log |z_n|}=+\infty.\]
We remark that  the sequence $\{|z_n|\}_{n=1}^{\infty}$ may be very sparse in $\mathbb{R}^+$.
Clearly, $\mathcal{TD}(f)$ is also nonempty and closed for transcendental $f$. Moreover, $\mathcal{TD}(f)=\mathcal{TD}(f+p)$
with any arbitrarily polynomial $p(z)$.\vskip 2mm
\par The following example is given to illustrate the notion of transcendental direction and Julia limiting direction.

\begin{example}\label{example}
Let $E_{\lambda}(z)=\lambda\exp(z)$, where $\lambda\in \mathbb{C}^{*}=\mathbb{C}\backslash \{0\}$. Then
\begin{enumerate}
\item $\mathcal{TD}(E_{\lambda})=[0, \frac{\pi}{2}]\cup[\frac{3}{2}\pi,2\pi)$;\vskip 1mm
\item $L(E_{\lambda})=[0, \frac{\pi}{2}]\cup[\frac{3}{2}\pi,2\pi)$ if $0\notin \mathcal{J}(E_{\lambda})$;\vskip 1mm
\item $L(E_{\lambda})=[0, 2\pi)$ if $0\in \mathcal{J}(E_{\lambda})$.
\end{enumerate}
\end{example}

This example inspired us that there may exist some relation between
$L(f)$ and $\mathcal{TD}(f)$, and indeed, this is our initial motivation for this paper. Our result holds for entire functions, and meromorphic functions with a direct tract. This class of meromorphic functions, were studied deeply in \cite{BRS}, have some similar dynamical behaviors which are very similar to entire functions. The function $f$ is said to have a direct tract if there exists a simply connected and unbounded domain $D$ and $R>0$ such that $f$ is holomorphic in $D$ and continuous on the closure of $D$, $|f(z)|>R$ for any $z\in D$ and $|f(z)|=R$ for $z\in \partial D$.

\begin{thm}\label{theorem-modification}
 Let $f$ be a transcendental entire function, or a transcendental meromorphic function with a direct tract. Then we have
\[\mathcal{TD}(f)\subseteq L(f).\]
Furthermore, if $\mu(f)<\infty$ and
$0<\delta(\infty,f)=1-\limsup_{r\to\infty}\frac{N(r,f)}{T(r,f)},$
 then
\begin{equation}\label{lower-bound-measure}
meas\left(L(f)\right)\geq \min\Big\{2\pi, \frac{4}{\mu(f)}\arcsin\sqrt{\frac{\delta(\infty, f)}{2}}\Big\}.
\end{equation}
\end{thm}
\vskip 0.5mm
Interestingly, even the weak growth behaviour along some unbounded sequence could be closely related to radial distribution of Julia set. Moreover, the lower bound of $meas\left(L(f)\right)$ is given for meromorphic functions with finite lower order, one direct tract and not so many poles. However, we do not know whether $\mathcal{TD}(f)\subseteq L(f)$ holds for all meromorphic functions.\vskip 2mm
\par Finally, there are also other mechanisms to even produce isolated Julia limiting direction for transcendental entire functions $f$ with finite order $\rho$. To illustrate the mechanism, we use an example which was suggested by Prof. Walter Bergweiler in the following theorem.
\begin{thm}\label{theorem-example}
There exists a transcendental entire function of finite order such that its Julia set has an isolated Julia limiting direction. Moreover, we can take
$$f(z)=z-\frac{1-\exp(-z)}{z(z^2+4\pi^2)},\quad \text{then}\quad  L(f)=[\frac{\pi}{2}, \frac{3\pi}{2}]\cup \{0\}.$$
\end{thm}

For $f(z)$ above, $\theta=0$ is the isolated Julia limiting direction. This example is motivated by the transcendental perturbation of  parabolic petals of the rational function $g(z)=z-\frac{1}{z^3}$ at $\infty$ to get the Baker domain for entire $f$. And the repelling axis for the parabolic petals will become the isolated Julia limiting direction.  This  phenomenon makes it difficult to answer the inverse problem for transcendental entire functions of finite order.\vskip 2mm
\par
 This paper is organised as follows. We obtain some basic properties of transcendental directions, and prove Theorem \ref{theorem-modification} in Section \ref{sect-Basic-Theorem-modification}. The proof of Theorem \ref{theorem-infinite-order}, that is the construction of suitable functions for the inverse problem, is given in Section \ref{sect-inverse-problem}. Theorem \ref{theorem-finite-order} and Theorem \ref{theorem-example}, further discussion on entire functions with finite order, are proved in Section \ref{sect-finite-order-case-entire} and Section \ref{sect-isolated-Julia-limiting-direction}, respectively.

\section{Basic Property of Transcendental Direction}\label{sect-Basic-Theorem-modification}

We first state a result due to Qiao \cite{qiao2}, which is very useful to deal with the case when there is an angular domain in $\mathcal{F}(f)$.
The result can be deduced from the proof of \cite[Lemma 1]{qiao2}.
\begin{lemma}\cite{qiao2}
Let $f$ be an analytic function in the angular domain $$\Omega(z_0,\theta,\delta)=\{z: |\arg (z-z_0)-\theta|<\delta\}.$$ Suppose that $f(\Omega(a,\theta,\delta))$ is contained in a simply connected hyperbolic domain in $\mathbb{C}$. Then
$$|f(z)|=O(|z|)^{\pi/\delta},\quad z\in\Omega(z_0,\theta,\delta')$$
for any $\delta'\in(0,\delta)$.
\end{lemma}
\par
  By Lemma 1, we establish the relation between transcendental directions and Julia limiting directions as follows.

\begin{prop}\label{prop-inclusion}
Let $f$ be a transcendental entire function, or a meromorphic function with direct tract. Then $\mathcal{TD}(f)\subseteq L(f)$.
\end{prop}
\begin{proof}
  We first treat the case that each component of $\mathcal{J}(f)$ is bounded. Note that every transcendental entire function must have direct tract. Then by \cite[Theorem 5.3]{BRS}, $\mathcal{F}(f)$ has Baker wandering domain $U$, that is, $\{U_n\}_{n=1}^{\infty}$ is the sequence of multiply connected Fatou components surrounding $0$ and $\lim_{n\rightarrow\infty} dist(0, U_n)=\infty$. Here, $U_n$ denotes the component of $\mathcal{F}(f)$ containing $f^n(U)$. For transcendental entire function, this fact is already proved in \cite{bak1} earlier. Moreover, by \cite[Theorem 5.1]{BRS}, for $n$ large enough, $U_n\subseteq T(U_{n+1})$, where we use the notation $T(X)$ to mean the
union of $X$ with its bounded complementary components.
It is easy to see $L(f)=[0, 2\pi)$. Or else, there must exists an angular domain $\Omega(a,\theta,2\epsilon)$ intersecting infinitely many   $U_n$, which is impossible. This implies $\mathcal{TD}(f)\subseteq L(f)$ directly.\vskip 2mm
\par Now, there is an unbounded component $\mathcal{J}_*$ in $\mathcal{J}(f)$. Let $W$ be the connected Fatou component containing  $\Omega(a,\theta,2\epsilon)$, and let $V$ be the Fatou component containing $f(W)$. Given $\theta\in \mathcal{TD}(f)$, we assume that $\theta\not\in L(f)$, otherwise there is nothing to prove.  Then we have $\Omega(a,\theta,2\epsilon)\subseteq \mathcal{F}(f)$ for $\epsilon>0$ and $a$ with $\arg a=\theta$. At the same time, there is a unbounded sequence $\{z_n\}\subseteq \Omega(a,\theta,2\epsilon)$ such that
\begin{equation}\label{eq-transcendental-growth}
\arg(z_{n})\to \theta\quad\text{and}\quad \frac{\log|f(z_n)|}{\log |z_n|}\to +\infty,\quad \text{as}\,\,n\to\infty.
\end{equation}
By (\ref{eq-transcendental-growth}), $V$ must be
unbound, then $V\subseteq \mathbb{C}\setminus T(\mathcal{J}_*)$. Followed by Lemma 1, we know that there exist positive constants $k$ and $A$, such that for $\epsilon'<\epsilon$,
$$|f(z)|\leq A |z|^{k},\quad \text{for}\quad z\in \Omega(a,\theta, 2\epsilon').$$
This contradicts with \eqref{eq-transcendental-growth}, so $\theta\in L(f)$. Hence, we also have $\mathcal{TD}(f)\subseteq L(f)$.
\end{proof}

The above proposition tells us that to measure $L(f)$, the possible way is to estimate $meas(\mathcal{TD}(f))$. Thus, we need to
find the direction on which $f$ grows faster than the polynomials. Recall Baerstein's result on the spread relation \cite{bae1}, it says that for $f$ without too many poles, $\log |f|$ is `enough large' on a substantial portion of circles.
\begin{lemma}\cite{bae1} \ Let $f$ be a transcendental meromorphic function with finite
lower order $\mu$ and positive $\delta(\infty,f)$. Let $\Lambda(r)$ be a positive function such that $\Lambda(r)=o(T(r,f))$ as $r\to\infty$,
and $D_{\Lambda}(r)=\{\theta\in
[0,2\pi): |f(re^{i\theta})|> e^{\Lambda(r)}\}$. Then there exists a positive, increasing and unbounded sequence $\{r_n\}$ such that
we have
\[\liminf\limits_{n\rightarrow\infty}mes(
D_{\Lambda}(r_n))
\ge\min\Big\{2\pi,\frac{4}{\mu}\arcsin
\sqrt{\frac{\delta(\infty,f)}{2}}\Big\}.\]
\end{lemma}
Baerstein's result inspires us to give the lower bound of $meas(\mathcal{TD}(f))$.
\begin{prop}\label{prop-lowner-bound}
Let $f$ be a transcendental entire or meromorphic function with finite lower order $\mu$ and $\delta(\infty,f)>0$,  then
$$meas(\mathcal{TD}(f))\geq \min\Big\{2\pi, \frac{4}{\mu}\arcsin\sqrt{\frac{\delta(\infty, f)}{2}}\Big\}.$$
\end{prop}
\begin{proof}
Let $\Lambda(r)=(T(r, f)\log r)^{\frac{1}{2}}$, and clearly $\Lambda(r)=o(T(r,f))$ as $r\to\infty$ since $f$ is transcendental. The value $\theta\in[0,2\pi)$ is called a $\Lambda$-type transcendental direction of $f$  if there exists an unbounded sequence $\{z_n\}$ with $|z_n|=r_n$ such that
$$\arg z=\theta,\quad \log|f(z_n)|\geq \Lambda(r_n).$$
We use $\mathcal{TD}_{\Lambda}(f)$ to denote the set of all $\Lambda$-type transcendental directions, and obviously $\mathcal{TD}_{\Lambda}(f)\subseteq \mathcal{TD}(f)$.
By Lemma 2, for any $\epsilon>0$, there exists a sequence of
$\{r_j\}_{j=1}^{\infty}$ tending to $\infty$ as $j\to\infty$, such that
\begin{equation}\label{ineq-spread-relation}
meas(D_{\Lambda}(r_j))\geq \min\Big\{2\pi, \frac{4}{\mu}\arcsin\sqrt{\frac{\delta(\infty, f)}{2}}\Big\}-\epsilon,
\end{equation}
where $D_{\Lambda}(r)=\{\theta\in
[0,2\pi): \log|f(re^{i\theta})|> \Lambda(r)\}$.\vskip 2mm
\par Next, we investigate the Lebesgue measure of
\[E:=\bigcap_{n=1}^{\infty}B_n\quad\, \text{with}\quad\, B_n=\bigcup_{j=n}^{\infty}D_{\Lambda}(r_j).\]
It is easy to see that $\{B_k\}_{k=1}^{\infty}$ is the monotone decreasing sequence of measurable sets in $[0, 2\pi)$ and
$meas(B_1)\leq 2\pi$. Then by monotone convergence theorem \cite[Theorem 1.19]{rudin}, we obtain
\begin{equation}\label{measure-E}
meas(E)=meas\Big(\bigcap_{n=1}^{\infty}
B_j\Big)=\lim_{n\rightarrow\infty} meas(B_n).
\end{equation}
Noting that $D_{\Lambda}(r_n)\subseteq B_n$ for each $n$, hence
\[\lim_{n\rightarrow\infty} meas(B_n)\geq \liminf_{n\rightarrow\infty} meas(D_{\Lambda}(r_n)).\]
Combining \eqref{ineq-spread-relation} and \eqref{measure-E} yields out
$$
meas(E)\geq\min\Big\{2\pi, \frac{4}{\mu}\arcsin\sqrt{\frac{\delta(\infty, f)}{2}}\Big\}-\epsilon.
$$
For each $\theta\in E$, there exists a sequence $\{r_{j_k}\}_{k=1}^{\infty}$ such that $\theta\in D_{\Lambda}(r_{j_k})$. This means that
$\theta\in \mathcal{TD}_{\Lambda}(f)$, which implies $E\subseteq \mathcal{TD}_{\Lambda}(f)$. Therefore, take $\epsilon$ arbitrarily small,
we get the desired conclusion immediately.
\end{proof}
\noindent{\bf Proof of Theorem \ref{theorem-modification}.}\,\,The observation $\mathcal{TD}(f)\subseteq L(f)$ comes from Proposition \ref{prop-inclusion}.
 Combining Proposition \ref{prop-inclusion} and Proposition \ref{prop-lowner-bound} yields the lower bound of $meas(L(f))$.
\vskip 2mm
 Finally, we give some application of Proposition \ref{prop-inclusion} on entire functions with slow growth.
 For transcendental entire function $f$ with order $\rho(f)\leq 1/2$, there exists a fact that $L(f)=[0,2\pi)$. It comes from the growth property of such functions (see \cite[p.696]{hell}): for $\rho<1/2$, define $K_r=\{\theta\in[0,2\pi):\log|f(re^{i\theta})|<r^{\rho}\}$, we either have a sequence $s_n\to\infty$ satisfying $L(s_n,f)>s_n^{\rho}$ where $L(r,f)=\min_{|z|=r}|f(z)|$, or we have $meas(K_r)\to 0$ as $r\in G\to\infty$ for some set $G$ of logarithmic density $1$.\vskip 2mm
 \par For meromorphic function $f$, we recall \cite[Chapter 5, Theorm 3.2]{gol1},
$$\limsup_{r\to\infty}\frac{\log^+\inf_{|z|=r}|f(z)|}{T(r,f)}\geq \frac{\pi\mu(f)}{\sin(\pi\mu(f))}[\delta(\infty,f)-1+\cos(\pi\mu(f))].$$
Specially when $\mu(f)=0$, the lower bound is just $\delta(\infty,f)$. Since $f$ is transcendental, the above inequality implies that if $\delta(\infty,f)>1-\cos(\pi\mu(f))$ and $\mu(f)<1/2$,
$$\limsup_{r\to\infty}\frac{\log^+\inf_{|z|=r}|f(z)|}{\log r}=\infty.$$
This means that $\mathcal{TD}(f)=[0,2\pi)$ under the above condition, so $L(f)=[0,2\pi)$ if $f$ has a direct tract.

\section{Inverse Problem for Julia Limiting Direction}\label{sect-inverse-problem}
In this section, we would construct the entire function of infinite order and meromorphic functions of any given order with preassigned Julia limiting directions.\vskip 2mm
\par For any non-empty  compact subset $\mathcal{E}$ of $[0,2\pi)$, there always exists a countable dense set $\{\theta_n,\, n\in\mathbb{N}\}\subseteq \mathcal{E}$. When $\mathcal{E}$ has only finitely many elements, we allow $\theta_n=\theta_k$ for all $n\geq k$, $k$ is some integer. The first step of the construction is to find
transcendental entire functions $f_n(z)$ with $\mathcal{TD}(f_n)=\{\theta_n\}$ for each $n$, and the second step is to
choose a sequence $\{a_n\}_{n=1}^{\infty}$ tending to zero rapidly and delicately  such that the function $S(z)=\sum_{n=1}^{\infty}a_n f_n(z)$  satisfies
\begin{enumerate}
\item $\{\theta_n,\,n\in\mathbb{N}\}\subseteq\mathcal{TD}(S)\subseteq L(S)$;\vskip 1mm
\item $\theta\notin L(S)$  for any $\theta \in [0,2\pi)\backslash \mathcal{E}$.
\end{enumerate}
Since the closure of $\{\theta_n,\,n\in\mathbb{N}\}$ is $\mathcal{E}$ and $L(S)$ is closed, we get $L(S)=\mathcal{E}$.\vskip 2mm
To complete the first step, we need an entire function
(of infinite order) with only one Julia limiting direction.  We first recall one function $E_0$ appeared in \cite{hay1}, which grows very fast in a strip, while it tends to zero outside.
\begin{lemma}\label{Hayman-Fuchs}\cite[Lemma 4.1]{hay1}
There exists an entire function $E_0(z)$, such that in the strip $\mathcal{A}_0=\{z: \mathcal{R}e z>0, -\pi\leq \mathcal{I}m z\leq \pi\}$, $E_0(z)=\exp(e^z+z)+O(z^{-2})$,
while outside $\mathcal{A}_0$,  $E_0(z)=O(z^{-2})$ uniformly as $z\to\infty$.
\end{lemma}
Obviously, $\mu(E_0)=\infty$. This function is constructed by Cauchy integral, and detailed discussion can be found in \cite[p.81-83]{hay1}. Next, we construct $f_n$ according to $\theta_n$ by using $E_0(z)$.
\begin{lemma}
There exist $R_0, \lambda_0>0$ and  $\lambda\in(0,\lambda_0]$ such that  $f_n(z)=\lambda e^{i\theta_n}E_0(e^{-i\theta_n}z)$ satisfies $\{z:z\notin \mathcal{A}_n=e^{i\theta_j}\mathcal{A}_0, |z|\geq R_0\} \subseteq \mathcal{F}(f_n)$ and $\mathcal{TD}(f_n)=L(f_n)=\{\theta_n\}$.
\end{lemma}
\begin{proof}
By Lemma \ref{Hayman-Fuchs}, there exist some positive constants $C$ and $R_0$ such that
\begin{equation}\label{asymptotic-behaviour}
\begin{split}
|f_n(e^{i\theta_n}z)-\lambda e^{i\theta_n}\exp(e^z+z)|& \leq \lambda C|z|^{-2},\,\quad\text{for}\,\, z\in \mathcal{A}_0\cap \{z:|z|\geq R_0\};\\
|f_n(z)|& \leq \lambda C|z|^{-2},\,\quad\text{for}\,\, z\in \{z:|z|\geq R_0\}\setminus \mathcal{A}_n.
\end{split}
\end{equation}
We could take $R_0$ large enough such that
\begin{equation}\label{setting-1}
C|z|^{-2}\leq C R_0^{-2}<1/3, \quad \text{for}\quad |z|\geq R_0,
\end{equation}
 and $\lambda_0$  sufficiently small to satisfy \[R'=\lambda_0(\exp\{e^{R_0}+R_0\}+1/3)<R_0.\]
By the maximum modulus theorem, and combining \eqref{asymptotic-behaviour}, we have
\begin{equation}\label{disk-estimation}
|f_{n}(z)|\leq M(R_0,f_n)\leq \lambda_0(\exp\{e^{R_0}+R_0\}+CR_0^{-2})<R',
\end{equation}
for all $z\in D(0,R_0)$, where $D(0,R_0)$ denotes the open disk with center point $0$ and radius $R_0$. This means that $f_n$ maps $D(0, R_0)$ to $D(0, R')$, and we get $D(0, R_0)\subseteq \mathcal{F}(f_n)$  by Montel's theorem. We see that $f_n$ maps
$\{z:|z|\geq R_0\}\setminus \mathcal{A}_n$ to $D(0, R_0)$. By the invariance property of Fatou set, we have
$$\{z: z\notin \mathcal{A}_n, |z|\geq R_0\}\subseteq  \mathcal{F}(f_n).$$
Thus,
$L(f_n)=\{\theta_n\}$. $\mathcal{TD}(f_n)=L(f_n)=\{\theta_n\}$ follows from $\emptyset\not=\mathcal{TD}(f_n)\subseteq L(f_n)$.
\end{proof}

 Now we use induction to choose the sequence $\{a_n\}_{n=1}^{\infty}$ with $a_1=1$. Suppose that $a_1,a_2,\cdots,a_k$ are already
chosen, we would take $a_{k+1}\in(0, 2^{-k}]$ such that $a_{k+1}f_{k+1}$ maps $D(0, R_0)$ to $D(0, 2^{-k}R')$ by \eqref{disk-estimation}. Furthermore, if
$\mathcal{A}_{k+1}\bigcap\big(\bigcup_{j=1}^{k}\mathcal{A}_j\big)\not=\emptyset$, then it must be bounded. We can choose $a_{k+1}$ small such that
\begin{equation}\label{construction-scaling}
|a_{k+1}f_{k+1}|\leq 2^{-k},\quad \text{for}\,\,z\in \mathcal{A}_{k+1}\bigcap\big(\bigcup_{j=1}^{k}\mathcal{A}_j\big).
\end{equation}
For the already chosen sequence $\{a_n\}_{n=1}^\infty$, we set
\begin{equation}\label{construction-limit-function}
S(z):=\sum_{n=1}^{\infty}a_n f_n(z)=\lim\limits_{k\rightarrow\infty}S_k(z),
\end{equation}
where $S_k:=\sum_{n=1}^{k}a_n f_n$.
Noting that $M(r, f_j)=M(r, f_1)$ for all $j$, and $M(r,f_j)$ denotes the maximum modulus of $f_j$ in $\{z:\,|z|=r\}$. We can get
\[|\sum_{j=n}^{m}a_jf_j(z)|\leq (\sum_{j=n}^{m}2^{-j})M(r,f_1)\]
for any $|z|\leq r$ and any $n<m$. This means that $S_k$  converges locally uniformly to $S$, thus $S$ is an  entire function.
\begin{prop}\label{prop-entire-infinite-order}
Let the function $S$ be defined as in \eqref{construction-limit-function}. Then $\mu(S)=\infty$ and
\[\mathcal{TD}(S)=L(S)=\mathcal{E}.\]
\end{prop}
\begin{proof}
For each $k\in\mathbb{N}$, from the choice way of $\{a_n\}_{n=1}^\infty$ and \eqref{asymptotic-behaviour}, it follows that
$$
|S_{k+1}(z)|\leq (\sum_{j=0}^{k}2^{-j})C|z|^{-2}\leq 2C R_0^{-2}
$$
uniformly on $\{z: |z|\geq R_0\,\,\text{and}\,\,z\notin \bigcup\limits_{j=1}^{k+1} \mathcal{A}_j\}.$ At the same time, we have
$$S_{k+1}\big(D(0,R_0)\big)\subseteq D(0,(\sum_{j=0}^{k}2^{-j})R')\subseteq D(0,2R').$$
Thus, after taking $R'$ with $2R'\leq R_0$,
$D(0,R_0)\subseteq \mathcal{F}(S_{k+1})$ by Montel's theorem again. From \eqref{setting-1} and the invariance property of Fatou set, we also have
\[\{z: z\notin \bigcup\limits_{j=1}^{k+1} \mathcal{A}_j,\,|z|\geq R_0\}\subseteq \mathcal{F}(S_{k+1}).\]
This implies $\mathcal{J}(S_{k+1})\subseteq \bigcup_{j=1}^{k+1}\mathcal{A}_j$, so
\begin{equation}\label{finite-step-limiting-direction}
L(S_{k+1})\subseteq\{\theta_1,\theta_2,\cdots,\theta_{k+1}\}.
\end{equation}
\par For any two distinct $m,j\in\{1,2,\cdots,k+1\}$, we take $z=re^{i\theta_j}$ then $z\notin \mathcal{A}_m$  if $r$ is  large enough.
Hence by \eqref{asymptotic-behaviour}, we have
$$|f_j(z)|\geq \lambda \exp(e^r-r)-\lambda Cr^{-2},\quad |f_m(z)|\leq \lambda Cr^{-2}.$$
Then we obtain
\begin{align*}
|S_{k+1}(z)|\geq |a_jf_j(z)|-\sum_{1\leq m\not=j\leq k+1}|a_mf_m(z)|\geq \lambda\big(2^{-j}\exp(e^r-r)-\sum_{n=0}^k2^{-n}Cr^{-2}\big).
\end{align*}
This leads to $\{\theta_j\}_{j=1}^{k+1}\subseteq \mathcal{TD}(g_{k+1})$.  Hence,  by Theorem \ref{theorem-modification} and \eqref{finite-step-limiting-direction}, we get
$$L(S_{k+1})=\mathcal{TD}(S_{k+1})=\{\theta_1,\theta_2,\cdots,
\theta_{k+1}\}.$$

 Similarly for the function $S$, we also consider the ray $\arg z=\theta_j$. Taking $z=re^{i\theta_j}$, we have
$$
|S(z)|\geq |a_jf_{j}(z)|-\sum_{n\neq j}|a_n f_n(z)|\geq \lambda(2^{-j}\exp(e^r-r)-2C r^{-2})
$$
for $r$  sufficiently large  by  \eqref{asymptotic-behaviour} and \eqref{construction-scaling}.
This implies that for each $j$, $\theta_j\in \mathcal{TD}(S)\subseteq L(S)$  and $\mu(S)=\infty$. Recall that $\mathcal{TD}(S)$ is closed,  then by Proposition \ref{prop-inclusion}, we deduce that
 \begin{equation}\label{one-sided-inclusion}
 \mathcal{E}=\overline{\{\theta_n,\,n\in\mathbb{N}\}}\subseteq \mathcal{TD}(S)\subseteq L(S).
 \end{equation}
 \vskip 0.5mm
\par Finally, for each $\theta\not\in \mathcal{E}$, we will prove $\theta\notin
L(S)$. When $\mathcal{E}=[0,2\pi)$, it is done. Thus we assume $\mathcal{E}\subsetneq [0,2\pi)$, so $[0,2\pi)\backslash \mathcal{E}$ is a union of at most countably many open intervals, and $\theta$ belongs to one above interval $I=(a,b)$ with $a,b\in \mathcal{E}$. Denote
\[\mathcal{A}_{\theta}(R, +\infty)=e^{i\theta}\mathcal{A}_0\bigcap \{z: |z|> R\}.\]
For any  $\epsilon>0$, there exists $\tilde{R}_{0}=\tilde{R}_0(\epsilon)$  such that for $R>R_0+\tilde{R}_0$,
\[\mathcal{A}_{\theta}(R, +\infty)\subseteq\{z:a+\epsilon\leq\arg z\leq b-\epsilon\},\]
which implies \[\mathcal{A}_{\theta}(R, +\infty)\bigcap \bigcup_{n=1}^{\infty}\mathcal{A}_{n}=\emptyset.\]
Then from the behavior of $|f_n|$ outside $\mathcal{A}_n$,  we have
\begin{equation}\label{estimate-S-Fatou-direction}
|S(z)|\leq \sum_{j=1}^\infty |a_nf_n(z)|\leq \sum_{j=1}^\infty 2^{-(j-1)}CR_0^{-2}=2CR_0^{-2}
\end{equation}
 for all $z\in \mathcal{A}_{\theta}(R,+\infty)$. Similarly as the mapping property of $S_{k+1}$, we also know
 $$S(D(0,R_0))\subseteq D(0,\sum_{k=0}^{\infty}2^{-k}R')\subseteq D(0,2R')$$
  by \eqref{disk-estimation}.
We take $R'$ in Lemma 4 satisfying $2R'<R_0$, then $D(0, R_0)\subseteq\mathcal{F}(S)$. Combining \eqref{setting-1} and \eqref{estimate-S-Fatou-direction} yield out $S(\mathcal{A}_{\theta}(R,+\infty))\subseteq D(0, R_0)$, further, we get
\[\mathcal{A}_{\theta}(R,+\infty)\subseteq \mathcal{F}(S),\quad \text{then}\quad \theta\notin L(S).\]

From the above facts together with \eqref{one-sided-inclusion}, we get that
 \[\mathcal{TD}(S)=L(S)=\mathcal{E}.\]
\end{proof}

 This proves the first part of Theorem \ref{theorem-infinite-order}. Next, we would construct the meromorphic function $f$ with order $\rho\in[0,\infty]$ and $L(f)=\mathcal{E}$. Note that all poles are located in the Julia set, the idea of the construction is very simple, we only need very careful arrangements of the poles to match with the given set $\mathcal{E}$.\vskip 2mm
\par At first, we derive a new sequence  $\{\hat{\theta}_k\}_{k=1}^{\infty}$ from $\{\theta_n,\,n\in\mathbb{N}\}$, such that for each $\theta_n$, $\sharp\{k:\hat{\theta}_k=\theta_n\}$ is countable and infinitely many. We choose $\{r_k \}_{k=1}^{\infty}$ as below
\begin{equation}\label{setting-r-k}
r_k=\left\{
\begin{array}{c}
\displaystyle (m_{k})^{\frac{1}{\rho}}, \quad\quad\quad\quad \rho\in(0,\infty),\\
\displaystyle (m_k)^k, \quad\qquad\quad\quad\quad \rho=0,\\
\displaystyle (m_k)^{\frac{1}{k}}, \quad\quad\quad\quad\quad \rho=\infty.
\end{array}
\right.
\end{equation}
where the sequence $\{m_k\}_{k=1}^\infty$ is given by
$$m_{k+1}=2^{\sum_{j=1}^{k}m_j}(k\in\mathbb{N}),\quad\text{and}\quad m_1=2.$$
Both $m_k$ and $r_k$ are monotone increasing in $k$, and tends to infinity as $k\rightarrow\infty$. Moreover, $\log _2 m_k\geq 2^{k-1}$ follows by induction. From \eqref{setting-r-k}, we have
\[r_{k+1}-r_{k}=\left\{
\begin{array}{c}
\displaystyle 2^{m_k/\rho}2^{\sum_{j=1}^{k-1}m_j/\rho}-m_k^{1/\rho}, \quad\quad\quad\qquad\qquad \rho\in(0,\infty),\\
\displaystyle 2^{(k+1) m_k}2^{(k+1)\sum_{j=1}^{k-1}m_j}-(m_k)^k, \quad\qquad\quad\quad\quad \rho=0,\\
\displaystyle 2^{m_k/(k+1)}2^{\sum_{j=1}^{k-1}m_j/(k+1)}-(m_k)^{1/k}, \quad\quad\quad\quad\,\, \rho=\infty.
\end{array}
\right.
\]
Thus, there exists the positive integer $k_0$ dependent on $\rho$ such that
$$r_{k+1}-r_{k}\geq 100,\quad\text{for}\quad k\geq k_0.$$
\vskip 1mm
\par Now, we denote $a_k=r_k \exp\{i\hat{\theta}_k\}$, and set
\begin{equation}\label{construction-limit-function-2}
g_0(z):=\sum\limits_{k=1}^{\infty}(z-a_k)^{-m_k}.
\end{equation}
For any  $z\not\in\{a_k\}_{k=1}^{\infty}$, let $$\delta(z):=\inf\{|z-a_k|,k=1,2,\cdots\},$$
and denote $k(z)$ denotes the positive integer dependent on $z$ with $|z-a_{k(z)}|=\delta(z)$.
From the choice of $a_k$, it is easy to see
\[\delta(z)\geq \max\{\big||z|-r_{k(z)}\big|,|\arg z-\hat{\theta}_{k(z)}|\}>0.\]
Then for $k>\max\{k_0,k(z)\}$ and any $p\in \mathbb{N}$, we have $r_k-|z|\geq 100+\delta(z)$, so
\begin{equation*}
\begin{split}
\Big|(z-a_k)^{-m_k}+\cdots+(z-a_{k+p})^{-m_{k+p}}\Big|&\leq (r_k-|z|)^{-k}+\cdots+(r_{k+p}-|z|)^{-(k+p)}\\
&\leq 100^{-k}+\cdots+100^{-(k+p)}.
\end{split}\end{equation*}
Thus, $\sum_{k=1}^n(z-a_k)^{-m_k}$ is locally uniformly convergent to $g_0(z)$. This means that $g_0(z)$ is meromorphic in $\mathbb{C}$.
\begin{prop}\label{prop-mermorphic-function}
Let $\rho\in[0,\infty]$, and $g_{\lambda}=\lambda g_0$, where $g_0$ is defined as in \eqref{construction-limit-function-2}. Then  there exists $\lambda>0$ such that $L(g_{\lambda})=\mathcal{TD}(g_\lambda)=\mathcal{E}$ and $\rho(g_{\lambda})=\rho$.
\end{prop}
\begin{proof}
For any  $k\geq k_0$, we consider the points on $|z|=r\in [r_k+2, r_{k+1}-2]$. From the choice of $\{r_n\}$, it is easy to see
$|z-a_j|\geq 2$ for $j=k,k+1$, and
\begin{equation*}\begin{split}
|z-a_j|&\geq r_j-r\geq r_j-r_{k+1}+r_{k+1}-r\geq 102,\quad \text{for}\,\,j\geq k+2,\\
|z-a_j|&\geq r_k-r_j+2\geq 102,\quad \text{for}\,\,j\leq k-1.
\end{split}\end{equation*}
Taking the estimation into $|g_0(z)|$ yields out
\[\sup_{|z|=r}|g_0(z)|\leq \sum_{j=k_0}^\infty 100^{-j}+2^{-1}+2^{-1}\leq 2.\]
Therefore, for $r\in [r_k+2, r_{k+1}-2]$ with $k\geq k_0$, we have $m(r,g_0)\leq \log 2$, so
\begin{equation}\label{Nevanlinna-1}
 \quad T(r,g_0)=N(r,g_0)+O(1);
\end{equation}
and
\begin{equation}\label{Nevanlinna-2}
N(r,g_0)=\int_{1}^{r}\frac{n(t,g_0)}{t}dt
\leq n(r,g_0)\log r=\Big(\sum_{j=1}^{k}m_j\Big)\log r.
\end{equation}
Combining \eqref{Nevanlinna-1} and \eqref{Nevanlinna-2}, we get
$$
T(r,g_0)\leq \Big(\sum_{j=1}^{k}m_j\Big)\log r+\log2,\quad r\in[r_k+2,r_{k+1}-2].
$$
\par Setting $\mathcal{X}=\cup_{k=1}^{\infty}(r_k-2,r_k+2)$. The above inequality means, for $\rho\in(0,\infty)$,
$$
\limsup_{r\rightarrow\infty;\, r\notin \mathcal{X}}
\frac{\log T(r,g_0)}{\log r}\leq
\limsup_{k\rightarrow\infty}\frac{\log \big(\sum\limits_{j=1}^{k}m_j\big)}{\log (r_k+2)}=\rho\limsup_{k\rightarrow\infty}\frac{\log \big(\sum\limits_{j=1}^{k}m_j\big)}{\log m_k}.
$$
The fact $\sum\limits_{j=1}^{k-1}m_j=\log_2 m_k$ implies
\begin{equation}\label{construction-limit-sequence}
\log\big(\sum\limits_{j=1}^{k}m_k\big)/\log m_k\to 1,\quad \text{as}\quad j\to\infty.
\end{equation}
Hence, we can conclude that
\begin{equation}\label{ineq-1}
\limsup_{r\rightarrow\infty;\, r\notin \mathcal{X}}
\frac{\log T(r,g_0)}{\log r}\leq \rho.
\end{equation}
While for any $r\in\mathcal{X}$, we note
\begin{equation}\label{ineq-2}
\limsup_{r\rightarrow\infty}\frac{\log T(r,g_0)}{\log r}\leq \limsup_{k\rightarrow\infty}\frac{\log T(r_k+2)}{\log (r_k-2)}=\limsup_{k\rightarrow\infty}\frac{\log T(r_k+2)}{\log (r_k+2)}.
\end{equation}
Combining the inequalities \eqref{ineq-1} and \eqref{ineq-2}, it follows that $\rho(g_0)\leq \rho$.\vskip 2mm
\par For $\rho=0$, consider the choice of $r_k=(m_k)^k$, we have $$
\limsup_{r\rightarrow\infty;\, r\notin \mathcal{X}}
\frac{\log T(r,g_0)}{\log r}\leq
\limsup_{k\rightarrow\infty}\frac{\log \big(\sum\limits_{j=1}^{k}m_j\big)}{k\log m_k}.
$$
Similarly as above, we can deduce that $\rho(g_0)=0$.\vskip 2mm
\par And also, we deduce that for $\rho\in(0,\infty]$,
\begin{equation*}
\begin{split}
T(2(r_k+1), g_0)&\geq N(2(r_k+1), g_0)\geq\int_{r_k+1}^{2(r_k+1)}\frac{n(t,g_0)}{t}dt\\
&\geq n(r_k+1, g_0)\log 2\geq \big(\sum_{j=1}^km_j\big)\log 2,
\end{split}
\end{equation*}
then by \eqref{construction-limit-sequence} again,
\[\lim_{k\rightarrow\infty}\frac{\log T(2r_k+2,g_0)}{\log (2r_k+2)}\geq \lim_{k\rightarrow\infty}\frac{\log \big(\sum\limits_{j=1}^{k}m_j\big)}{\log (r_k+2)}\geq \rho.\]
This implies that for $\rho\in(0,\infty]$, $\rho(g_0)\geq \rho$. Therefore, for $\rho\in[0,\infty]$, $g_0$ defined in \eqref{construction-limit-function-2} must have $\rho(g_0)=\rho$.\vskip 2mm
 \par Since $g_0$ is holomorphic on $D(0,1)$, by maximum modulus theorem, there is a constant $M>0$ such that
$\sup_{|z|\leq 1}|g_0(z)|\leq M$.
Thus, we can take a positive constant $\lambda\leq \min\{1,\frac{1}{2M}\}$ such that $g_{\lambda}$ maps $D(0,1)$ to $D(0,1/2)$. This implies $D(0,1)\subseteq\mathcal{F}(g_{\lambda})$.  It is not difficult to check
\[|g_0(z)|\leq \sum_{k=1}^\infty 2^{-m_k}< 1,\quad \text{for}\quad z\in \bigcap_{k=1}^{\infty}\{z: |z-a_k|\geq 2\}.\]
It means $\bigcap_{k=1}^{\infty}\{z: |z-a_k|\geq 2\}\subseteq \mathcal{F}(g_{\lambda})$, thus we get
\[\mathcal{J}(g_{\lambda})=\mathbb{C}\setminus\mathcal{F}(g_{\lambda})\subseteq \bigcup_{k=1}^{\infty}D(a_k,2).\]
\par Recall that for each $j$, $a_j$ is the pole of $g_{\lambda}$ with $\arg a_j\in \{\hat{\theta}_k\}_{k=1}^\infty$, and $a_j\in\mathcal{J}(g_{\lambda})$. And for each $\theta_n$, $\sharp\{k:\hat{\theta}_k=\theta_n\}=\infty$,
and $\lim\limits_{j\rightarrow\infty}|a_j|=\infty$. Thus, both $L(g_{\lambda})$ and $\mathcal{TD}(g_\lambda)$ contain  $\{\theta_n,\,n\in\mathbb{N}\}$, which implies $\overline{\{\theta_n,\,n\in\mathbb{N}\}}=\mathcal{E}\subseteq L(g_{\lambda})\bigcap\mathcal{TD}(g_\lambda)$.\vskip 2mm
\par For any $\theta\not\in \mathcal{E}$, we set \[d(\theta):=\inf\{|\theta-\theta_n|, \quad n=1,2,\cdots\},\]
then $d(\theta)>0$. We now claim that
$\theta\not\in L(g_{\lambda})$ if $\theta\not\in \mathcal{E}$. Otherwise, we assume that $\theta\in L(g_{\lambda})$, then the ray $\arg z=\theta$ must intersect with infinitely many disks $D(a_{k_j},2)$ by $\mathcal{J}(g_\lambda)\subseteq\bigcup_{k=1}^{\infty}D(a_k,2)$. This means that the distance between $a_{k_j}$ and the ray $\arg z=\theta$ is at most $2$. By the simple geometric observation, it follows that
\[r_{k_j}\sin|\theta-\hat{\theta}_{k_j}|\leq 2.\]
Since $r_{k_j}\to \infty$ as $j\to\infty$, the above inequality implies $|\theta-\hat{\theta}_{k_j}|\to 0$ as $j\to\infty$, which
contradicts with $d(\theta)>0$. Thus, if $\theta\not\in\mathcal{E}$, then $\theta\not\in L(g_{\lambda})$, which means $L(g_{\lambda})\subseteq \mathcal{E}$.
This fact and $\mathcal{E}\subseteq L(g_{\lambda})$ leads to $L(g_{\lambda})=\mathcal{E}$.\vskip 2mm
\par For $\theta\not\in\mathcal{E}$, we consider the angle $\Omega(\alpha,\beta)=\{z: \arg z\in(\alpha,\beta)\}$ with \[\alpha=\theta-\frac{d(\theta)}{2},\quad \beta=\theta+\frac{d(\theta)}{2}.\]
From the geometric observation, it follows that for $z\in \Omega(\alpha,\beta)$
\[|z|\sin\frac{d(\theta)}{2}\leq |z-a_n|,\quad n\in\mathbb{N}.\]
There exists a constant $r_0>0$ such that $|z|\sin(d(\theta)/2)\geq 100$ when $|z|\geq r_0$. Taking these estimate into \eqref{construction-limit-function-2} yields out that
\[|g_\lambda(z)|\leq \lambda\sum_{k=1}^\infty 100^{-j}\leq 1,\quad\text{for}\,\,z\in\Omega(\alpha,\beta)\bigcap \{z: |z|>r_0\},  \]
which implies $\theta\not\in\mathcal{TD}(g_\lambda)$. This leads $\mathcal{TD}(g_\lambda)\subseteq \mathcal{E}$, so similarly $\mathcal{TD}(g_\lambda)=\mathcal{E}$.
\end{proof}
\noindent {\bf Proof of Theorem \ref{theorem-infinite-order}.} Clearly, the entire function with infinite lower order and the meromorphic function with order $\rho\in[0,\infty]$ are given by Proposition \ref{prop-entire-infinite-order} and Proposition \ref{prop-mermorphic-function} respectively.
\vskip 2mm
\begin{remark} Proposition \ref{prop-entire-infinite-order} and Proposition \ref{prop-mermorphic-function} also guarantee that there is one analogue of Theorem 1 for the set of all transcendental directions.
\end{remark}

 \section{The Example on $L(f)$ being the Union of Disjointed Intervals}
 \label{sect-finite-order-case-entire}
In this section, we will prove Theorem \ref{theorem-finite-order}, which is one try on the inverse problem of Julia limiting direction for  transcendental entire functions with finite order. The idea is similar as in the proof of Theorem \ref{theorem-infinite-order}, and we need find an entire function $f$ of order $\rho$ such that $L(f)$ is a closed interval. This is stated in Proposition \ref{prop-entire-finite-order} below.
\begin{prop}\label{prop-entire-finite-order}
For any given $\rho\in(1/2,\infty)$, both $r,\epsilon\in(0,1)$ and the closed interval $J\subseteq [0, 2\pi)$ with $meas(J)\geq \pi/\rho$, then there always exists  transcendental entire function $f$ satisfying
\begin{enumerate}\label{prop-finite-order}
\item $\rho(f)=\rho$ and  $\mathcal{TD}(f)=J$;\vskip 1mm
\item $f$ maps $D(0, 1)$ to $D(0, r)$;\vskip 1mm
\item for any $\theta\in [0, 2\pi)\backslash J$, there exists $\delta>0$ dependent on $\theta$ and $R>0$ dependent on $\theta,\varepsilon$ such that $|f(z)|< \epsilon$ for all $z\in\Omega(\theta-\delta,\theta+\delta)\cap\{z: |z|>R\}$.
\end{enumerate}
\end{prop}
\par This time, for the convenience of readers, we first assume Proposition \ref{prop-finite-order} to be true and prove Theorem \ref{theorem-finite-order}, then give the proof of the proposition later.\vskip 2mm
\noindent{\bf Proof of Theorem \ref{theorem-finite-order}.}\,\,
If $m=1$, denote $J=I_1$, and we claim that $f$ constructed in Proposition \ref{prop-finite-order} is just what we want.
By Proposition \ref{prop-finite-order} and Theorem \ref{theorem-modification}, $\rho(f)=\rho$ and $J=\mathcal{TD}(f)\subseteq L(f)$.
For any $\theta\notin J$, by property (3) in Proposition \ref{prop-entire-finite-order},
there exists $\delta>0$ and $R>0$ such that for any $z\in \Omega(\theta-\delta,\theta+\delta)\cap\{z: |z|>R\}$, we have $f(z)\subseteq D(0, 1)$. This implies $z\in \mathcal{F}(f)$
by Montel's theorem, so $\theta\notin L(f)$, which means $L(f)\subseteq J$. Therefore, $L(f)=J$.\vskip 2mm
\par
If $m\geq 2$, for $k=1,2,\cdots,m$, we denote $\epsilon_k=r_k=2^{-(k+1)}$. By Proposition \ref{prop-finite-order}, there exists entire functions $f_k$ such that\vskip 1mm
\begin{enumerate}
\item $\rho(f_k)=\rho$ and  $\mathcal{TD}(f_k)=I_k$;\vskip 1mm
\item $f_k$ maps $D(0, 1)$ to $D(0, r_k)$;\vskip 1mm
\item For any $\theta\in [0, 2\pi)\backslash I_k$, there exists $\delta_k(\theta)>0$ and $R_k(\theta,\epsilon_k)>0$ such that $|f_k(z)|< \epsilon_k$ for all $z\in\Omega(\theta-\delta,\theta+\delta)\cap\{z: |z|>R\}$.
\end{enumerate}
Set $f=\sum_{k=1}^{m}f_k$, we know that $f$ maps $D(0,1)$ in $D(0,1-2^{-m})$. Consider property (1) and (3) of $f_k$, it is easy to see from Theorem \ref{theorem-modification} that
 \[\bigcup_{k=1}^{m}I_k=\bigcup_{k=1}^{m}\mathcal{TD}(f_k)=\mathcal{TD}(f)\subseteq L(f).\]
Similarly, for each $\theta\not\in \bigcup_{k=1}^{m}I_k$, again by property (3) of $f_k$, we have $|f(z)|\leq 1-2^{-m}$ for
all $z\in \Omega(\theta-\delta_0,\theta+\delta_0)\cap\{z: |z|>R_0\}$, where
\[\delta_0=\min\{\delta_k(\theta),k=1,2,\cdots,m\}\,\,\text{and}\,\,R_0=\max\{R_k(\theta,\epsilon_k),k=1,2,\cdots,m\}.\]
Clearly, $\Omega(\theta-\delta_0,\theta+\delta_0)\cap\{z: |z|>R_0\}$ belongs to $\mathcal{F}(f)$, so $\theta\not\in L(f)$, then
$\bigcup_{k=1}^{m}I_k=L(f)$ immediately. This completes the proof of Theorem \ref{theorem-finite-order}.\vskip 2mm
\par
Now, we would prove Proposition \ref{prop-finite-order} by using Mittag-Leffler function $E_{\alpha}$ given by
\[
E_\alpha(z)=\sum^\infty_{n=0} \frac{z^n}{\Gamma(\alpha^{-1}n+1)},\quad 0<\alpha<\infty.
\]
Moreover, it has the uniform asymptotic behavior \cite[(5.38) and (5.40')]{gol1},
\begin{equation}\label{Mittag-Lefller-asymptotic}
E_{\alpha}(z)=\left\{
\begin{array}{c}
\displaystyle \alpha\exp\left(z^{\alpha}\right) + O\left(|z|^{-1}\right), \,\quad\quad\quad |\arg (z)| \leq \frac{\pi}{2\alpha},\\
\displaystyle O\left(|z|^{-1}\right), \quad\quad\quad\qquad\qquad \frac{\pi}{2\alpha}< |\arg (z)| \leq \pi.
\end{array}
\right.
\end{equation}
Clearly, $\rho(E_{\alpha})=\alpha$ and $\mathcal{TD}(E_{\alpha})=[-\frac{\pi}{2\alpha},\frac{\pi}{2\alpha}]$.\vskip 2mm
\noindent {\bf Proof of Proposition \ref{prop-finite-order}.}\,\,Property (2) is just stated to show $D(0,1)\subseteq\mathcal{F}(f)$, but it is not essential. In fact, if there is an entire function $F$ satisfying property (1) and property (3), then we can take $f=\lambda F$
with $\lambda<rM(1,F)^{-1}$. This function $f$ satisfies all property (1)-(3). Thus, in the following discussion, we only need to construct the entire function with property (1) and property (2), which we call the $J$-properties for simplicity.\vskip 2mm
\par Set $l=meas(J)$. We will consider two cases according to $l\rho/\pi\in\mathbb{N}$ or not.\vskip 2mm
\par\noindent Case 1.\,\,We assume that $l\rho/\pi\in\mathbb{N}$, that is, $l=n\pi/\rho$ for some $n\in \mathbb{N}$. Without lose of generality, we assume that $J=[-\frac{\pi}{2\rho}, \frac{(2n-1)\pi}{2\rho}] (mod\, 2\pi)$. If not, we just consider a linear conjugation of $f$. For $n=1$, recall (\ref{Mittag-Lefller-asymptotic}), $E_{\rho}(z)$ is the suitable function.  For $n\geq 2$, we decompose $J$ into $n$
 intervals $J_1,J_2,\cdots,J_n$ of measure $\pi/\rho$ such that $J=\bigcup_{k=1}^{n}J_{k}$ and $int J_{i}\bigcap int J_{j}=\emptyset$ for distinct $i$ and $j$. And for each $1\leq k\leq n$, there exists $f_k(z)=\lambda_kE_{\rho}(\lambda_{k}^{-1}z)$ satisfying $J_k$-properties, where $|\lambda_k|=1$. Again by (\ref{Mittag-Lefller-asymptotic}), it is easy to see that $f(z)=\sum_{k=1}^{n} f_k(z)$ satisfies the desired $J$-properties.\vskip 2mm
\par\noindent Case 2.\,\,Now, $l\rho/\pi\notin \mathbb{N}$. Denote $L=[l\rho/\pi]$, where $[x]$ be the integer part of $x\geq 0$.
The condition $meas(J)\geq \pi/\rho$ means $1\leq L\leq [2\rho]$. Without loss of generality, we assume that $J=[-\frac{\pi}{2\rho}, l-\frac{\pi}{2\rho}]$, and decompose it as $J=\bigcup_{k=1}^{L+1}J_k$ where
\[J_k=\big[(2k-3)\frac{\pi}{2\rho},(2k-1)\frac{\pi}{2\rho}\big](k=1,2,\cdots,L),\quad J_{L+1}=\big[(2L-1)\frac{\pi}{2\rho},l-\frac{\pi}{2\rho}\big].\]
For every $J_k(k=1,2,\cdots,L-1)$, we take $f_k(z)=\lambda_kE_{\rho}(\lambda_{k}^{-1}z)$ satisfying $J_k$-properties as in Case 1.
Next, we would construct one entire function $f_L$ satisfying $\tilde{J}=J_{L}\cup J_{L+1}$ properties, and $\tilde{l}=meas(\tilde{J})\in(\pi/\rho,2\pi/\rho)$. \vskip 2mm
 \par We simplify $\tilde{J}$ to $\tilde{J}=[-\frac{\pi}{2\rho},\tilde{l}-\frac{\pi}{2\rho}]$ by one linear conjugation of $f_L$. Set $\theta_0=l-\frac{\pi}{\rho}$,
 and denote \[u_1(x)=\chi_{(-\frac{\pi }{2\rho}, \frac{\pi }{2\rho})}(x)\cos\rho x, \quad u_2(x)=u_1(x+\theta_0),\]
where $\chi_{(-\frac{\pi }{2\rho}, \frac{\pi }{2\rho})}(x)$ is the characteristic function of $(-\frac{\pi }{2\rho}, \frac{\pi }{2\rho})$. By \eqref{Mittag-Lefller-asymptotic}, we have
\begin{equation}\label{Mittag-Leffler-growth}
\lim_{r\rightarrow\infty}\frac{\log^{+} |E_{\alpha}(re^{i\beta})|}{ r^{\rho}}=u_1(\beta),\quad \text{for}\,\,\beta\in \mathbb{R}.
\end{equation}
It is not difficult to deduce that
if $u_1(x)\neq 0$ or $u_2(x)\neq 0$, then $u_1(x)-u_2(x)\neq 0$
with at most one exceptional point. This fact, \eqref{Mittag-Lefller-asymptotic} and \eqref{Mittag-Leffler-growth} imply that $f_L(z)=E_{\rho}(z)+e^{-i\theta_0}E_{\rho}(e^{i\theta_0}z)$ grows as
\[|f_L(z)|=\rho\exp\big\{\max\{u_1(\beta),u_2(\beta)\}
|z|^{\rho}\big\}(1+o(1))\]
in the angular domain $\Omega(-\frac{\pi}{2\rho}, l-\frac{\pi}{2\rho})$ with at most one exceptional direction. At the same time,
$|f_L(z)|=O(|z|^{-1})$ outside the above angle. Since $\mathcal{TD}(f_L)$ is closed, then the above argument means that $f_L(z)$
is the entire function satisfying $\tilde{J}$ properties. Therefore, $f(z)=\sum_{k=1}^{L} f_k(z)$ is the function with the desired property.

\section{Entire Function with Isolated Julia Limiting Direction}\label{sect-isolated-Julia-limiting-direction}
To prove Theorem \ref{theorem-example}, we just need to show that the entire function
\begin{equation}\label{example-function}f(z)=z-\frac{1-\exp(-z)}{z(z^2+4\pi^2)}
\end{equation}
has one isolated Julia limiting direction.
It is easy to see that all fixed points of $f$ are $z=2k\pi i\,(k\in\mathbb{Z}\setminus\{-1,0,1\})$, $f(\mathbb{R})\subseteq\mathbb{R}$, $f(x)<x$ for all $x\in \mathbb{R}$. Since $f$ has no finite fixed point in $\mathbb{R}$, $f^n(x)$ must diverge to $-\infty$ as $n$ goes to infinity. Furthermore, we have
\begin{equation} \label{asymptotic-f}
\lim_{n\rightarrow\infty}|f^{n+1}(x)|\exp\left\{-\frac{|f^n(x)|}{2}\right\}=\infty.
\end{equation}
\par For any $L>0$,  we denote \[U_{L}=\{z: Re z\geq L, Im z\geq L\}, \quad V_{L}=\{z: Re z\geq L, Im z\leq -L\}.\]
Given any $R\geq L$, we define two curves
\[\gamma_{1, R}(t)=(R+t)+R\sqrt{-1},\quad \gamma_{2, R}(t)=R+(R+t)\sqrt{-1}\] for all $t\geq 0$. It is obvious that
\begin{equation}
U_{L}=\bigcup_{R\geq L}\big(\gamma_{1, R}\cup \gamma_{2, R}\big).
\end{equation}
\par Next, we first state one technical lemma, which is useful to prove Theorem \ref{theorem-example}. The proof of this lemma, containing some long computation, will be shown in the end of this section.
\begin{lemma}\label{lemma-baker-domain-estimation}
Under the setting above, there exists a positive constant $L_0$ such that if $L\geq L_0$, we have
\begin{enumerate}
\item $Re f(\gamma_{1,R}(t))\geq R+\frac{1}{50} R^{-3}$, while $Im f(\gamma_{1,R}(t))\geq R+\frac{1}{10} R^{-3}$ for $t\leq \frac{R}{2}$ and $Im f(\gamma_{1,R}(t))\geq R+\frac{1}{5}t^{-3}$  for $t\geq \frac{R}{2}$;\vskip 1mm
\item $Im f(\gamma_{2,R}(t))\geq R+\frac{1}{50} R^{-3}$, while $Re f(\gamma_{2,R}(t))\geq R+\frac{1}{10} R^{-3}$ for $t\leq \frac{R}{2}$ and $Re f(\gamma_{1,R}(t))\geq R+\frac{1}{5}t^{-3}$  for $t\geq \frac{R}{2}$.
\end{enumerate}
\end{lemma}

\par In addition, we still need one Baker's result to guarantee that $U_L$ and $V_L$ are contained in two different Fatou components.

\begin{lemma}[\cite{bak2}]\label{lemma-baker-domain}
Let $f$ be a transcendental entire function. Assume $U$ be a Baker domain, then for any compact set $K\subseteq U$, there exists $c=c(K)$ and $N\in \mathbb{N}$ such that $|f^n(z_1)|\leq |f^n(z_2)|^{c}$ for any $z_1, z_2 \in K$ and $n\geq N$.
\end{lemma}

\par\noindent{\bf Proof of Theorem \ref{theorem-example}.}\,\,From Lemma \ref{lemma-baker-domain-estimation}, we directly know
that $f(U_{L})\subset (U_{L})$ and $f(V_{L})\subset (V_{L})$. This implies that the Fatou component containing $U_L$ and that containing $V_L$ must be invariant Fatou component, and can not be attraction Basin, parabolic Basin and Siegel disk.
By the theorem on classifications of Fatou components (\cite{ber1} and \cite[Theorem 2.1]{sch}), the only possibility is Baker domain.\vskip 2mm
 \par Now, we exclude the case that $U_{L}$ and $V_{L}$ are contained in the same Baker domain $U$. Otherwise, there must be some compact interval $J\subseteq \mathbb{R}$ which belongs to this Baker domain. For any $x\in J$, $f(x)\in \mathbb{R}\cap U$. Let $K=\{x, f(x)\}$, it is compact subset of $U$. Applying Lemma \ref{lemma-baker-domain} to $z_1=x,z_2=f(x)$, we have $|f^{n+1}(x)|\leq |f^{n}(x)|^{c}$ for some positive $c$, which contradicts with (\ref{asymptotic-f}) since $\lim\limits_{n\rightarrow\infty}f^n(x)=-\infty$. Thus, $U_L$ and $V_L$ must be contained in the different Baker domains, and the ray $\arg z=0$ must be the Julia limiting direction.\vskip 2mm
\par Since both $U_L$ and $V_L$ belong to $\mathcal{F}(f)$, we know that
$(0, \frac{\pi}{2})\bigcup (-\frac{\pi}{2}, 0) (mod \ 2\pi)$ is in the complement of $L(f)$. On the other hand, by Theorem \ref{theorem-modification}, clearly
$[\frac{\pi}{2}, \frac{3\pi}{2}]\subseteq \mathcal{TD}(f)\subseteq L(f)$. Thus, the above argument leads
\[L(f)=\big[\frac{\pi}{2}, \frac{3\pi}{2}\big]\cup \{0\}.\]
\par The proof of Theorem \ref{theorem-example} is therefore complete if we can establish Lemma \ref{lemma-baker-domain-estimation}.
\vskip 2mm
\noindent{\bf Proof of Lemma \ref{lemma-baker-domain-estimation}.} We write $f(z)=g_1(z)+g_2(z)+g_3(z)$,  where \[ g_1(z)=z-\frac{1}{z^3},\,\, g_{2}(z)=\frac{\exp(-z)}{z^3},\,\,g_3(z)=\frac{1-\exp(-z)}{z^3}\frac{4\pi^2}{z^2+4\pi^2}.\] For $R\geq L$, we always have
\begin{equation}\label{estimation-g-2}
|g_2(\gamma_{1, R}(t)|\leq R^{-3}\,e^{-(R+t)},\quad |g_2(\gamma_{2, R}(t)|\leq R^{-3}\,e^{-R}.
\end{equation}
By using the above estimate, there exists $L_0$ such that $L\geq L_0$, we have
\[
|g_3(\gamma_{i, R}(t)|\leq 8\pi^2((R+t)^2+R^2)^{-\frac{5}{2}},\quad (i=1,2)
\]
which implies
\begin{equation}\label{estimation-g-3}
\begin{split}
|g_3(\gamma_{i,R}(t))|&\leq 25 R^{-5},\qquad\quad\quad\text{for}\,\,t\in[0,R/2];\\
|g_3(\gamma_{i,R}(t))|&\leq 98\,t^{-5},\qquad\quad\quad\text{for}\,\, t\in[R/2,+\infty).
\end{split}
\end{equation}
On the other hand, denote $z=x+\sqrt{-1}\,y$, the straightforward calculation yields
\[
g_1(z)=x+\frac{3xy^2-x^3}{(x^2+y^2)^3}+\sqrt{-1}\Big(\frac{3x^2y-y^3}{(x^2+y^2)^3}+y\Big).
\]
We can deduce that for $t\in [0, R/2]$,
\begin{equation}\label{estimates-1}
\begin{split}
Re g_1(\gamma_{1,R}(t))-R & \geq \frac{(R+t)(3R^2-9R^2/4)}{(9R^2/4+R^2)^{3}}
\geq \frac{1}{46R^3}, \\
Im g_1(\gamma_{1,R}(t))-R & \geq \frac{R(27R^2/4-R^2)}{(9R^2/4+R^2)^3}\geq
\frac{1}{6R^3},
\end{split}
\end{equation}
while for $t\in [R/2,\infty)$,
\begin{equation}\label{estimates-2}
\begin{split}
Re g_1(\gamma_{1,R}(t))-R&\geq \frac{R}{2}-\frac{(R+t)^3}{(R+t)^{6}}\geq \frac{R}{2}-\frac{16}{81R^4}\geq \frac{R}{3}, \\
Im g_1(\gamma_{1,R}(t))-R & \geq \frac{3Rt^2}{8(R+t)^6}\geq \frac{1}{4 t^3}.
\end{split}
\end{equation}
Similarly, we can also obtain the estimation of $Re g_1(\gamma_{2,R}(t))$ and $Im g_1(\gamma_{2,R}(t))$ as
\begin{equation}\label{estimates-3}
\begin{split}
&Re g_1(\gamma_{2,R}(t))-R\geq \frac{1}{6}R^{-3},\quad Im g_1(\gamma_{2,R}(t))-R\geq \frac{1}{46}R^{-3},\quad \text{for}\,\,t\in[0,\frac{R}{2}];\\
&Re g_1(\gamma_{2,R}(t))-R\geq \frac{1}{4}t^{-3},\quad Im g_1(\gamma_{2,R}(t))-R\geq \frac{1}{3}R,\quad\quad \text{for}\,\,t\in[\frac{R}{2},\infty].
\end{split}
\end{equation}
Combining these estimates \eqref{estimation-g-2}, \eqref{estimation-g-3}, \eqref{estimates-1}, \eqref{estimates-2} and \eqref{estimates-3} together, this completes the proof of Lemma \ref{lemma-baker-domain-estimation}.

 \vskip 4mm

\noindent{\bf Acknowledgement}: The authors would like to thank Prof. Walter Bergweiler's great help and Prof. Jianyong Qiao's encouragement during the preparation for this paper. We also thank Prof. Jianhua Zheng for his comments on the early draft of this paper. And this work was supported by the National Natural Science Foundation of China (Grant No.11771090, No.11571049) and Natural Sciences Foundation of Shanghai (Grant No.17ZR1402900).

\end{document}